\documentclass[11pt,leqno]{amsproc}
\usepackage{amsmath,amsthm,amssymb,amsfonts}
\usepackage{graphicx}
\usepackage{color}
\usepackage[active]{srcltx}

\newtheorem{theorem}{Theorem}[section]
\newtheorem{lemma}[theorem]{Lemma}

\newtheorem{proposition}[theorem]{Proposition}
\newtheorem{corollary}[theorem]{Corollary}
\theoremstyle{definition}
\newtheorem{definition}[theorem]{Definition}

\numberwithin{equation}{section}

\def\dist{{\rm dist}}

\def\N{{\mathbb N}}
\def\K{{\mathbb K}}

\def\R{{\mathbb R}}

\newcommand{\Rea}{\text{Re}\, }
\def\C{{\mathbb C}}

\def\llll{{\longrightarrow}}

\def\sep{{ \ \  }}

\def\seg{{\ \ \ \  \ \  }}
\def\sem{{\ \ \ \ \ \  }}

\DeclareMathOperator{\rea}{Re}

\def\com#1{{``#1''}}

\title[The Bishop-Phelps-Bollob\'{a}s   and approximate  hyperplane
series properties] {The Bishop-Phelps-Bollob\'{a}s and  approximate  \\ hyperplane  series
properties}

\author[M.D. Acosta]{Mar\'{\i}a D. Acosta}
\address{Universidad de Granada, Facultad de Ciencias,
Departamento de An\'{a}lisis Matem\'{a}tico, 18071 Granada, Spain}
\email{dacosta@ugr.es}
\author[M.~Masty{\l}o]{Mieczys{\l}aw  Masty{\l}o}
\address{Faculty of Mathematics and Computer Science,
A. Mickiewicz University, Umultowska 87, 61-614 Pozna\'n, Poland}
\email{mastylo@amu.edu.pl}
\author[M. Soleimani]{Maryam Soleimani-Mourchehkhorti}
\address{Department of Mathematics, University of Isfahan, Isfahan, Iran, 81745-163}
\email{m.soleymanei@sci.ui.ac.ir}

\thanks{The  first  author was  supported  by MTM2015-65020-P, Junta de
Andaluc\'{\i}a P09-FQM--4911 and FQM--185. The second  author was supported by National Science
Center, Poland, project no.  2015/17/B/ST1/00064.}



\begin{document}
\begin{abstract}
We study the Bishop-Phelps-Bollob\'as property for operators between Banach spaces.
Sufficient conditions are given for generalized direct sums of Banach spaces with respect to
a~uniformly monotone Banach sequence lattice to have the approximate hyperplane series property.
This result implies  that Bishop-Phelps-Bollob\'as theorem holds for operators from $\ell_1$ into
such
direct sums of Banach spaces. We also show that the direct sum of two spaces with the approximate hyperplane
series property  has such property whenever the norm of the direct sum is absolute.
\end{abstract}

\maketitle

\baselineskip=.65cm

\section{Introduction}

 The motivation for this paper comes from recent intensive study of the famous
Bishop-Phelps Theorem \cite{BP}, which states that every Banach space is
subreflexive, i.e., the set of norm attaining (continuous and linear) functionals
on a~Banach space is dense in its topological dual.

The first who initiated the study of the denseness of
norm-attaining operators between two Banach spaces was
Lindenstrauss \cite{Lin}. Later a~lot of attention was devoted to
extend Bishop-Phelps result in the setting of operators on Banach
spaces (see, e.g., \cite{AcR, Bou}).

In 1970, Bollob\'as showed the following \com{quantitative
version} which is now called Bishop-Phelps-Bollob\'as Theorem
\cite{Bol}. To state this result we mention that for a~normed
space $X$, we denote by $B_X$ and $S_X$ the closed unit ball and
the unit sphere of $X$, respectively. As usual, $X^{*}$ denotes
the dual Banach space of $X$.

The mentioned above version of the Bishop-Phelps-Bollob\'{a}s Theorem from \cite[Theorem 16.1]{BoDu} states
that if $X$ is a~Banach space and $0<\varepsilon <1$, then given $x \in B_X$ and $x^\ast \in S_{X^\ast}$ with
$\vert 1 - x^\ast(x ) \vert < \varepsilon^{2}/4$, there are elements $y\in S_X $ and $ y^\ast \in S_{X^\ast}$
such that $y^\ast (y)=1$, $\Vert y-x \Vert < \varepsilon$ and $ \Vert y^\ast -x^\ast \Vert < \varepsilon $.

For a refinement of the above result see \cite[Corollary 2.4(a)]{CKMMR}. In 2008 Acosta, Aron, Garc\'{\i}a
and Maestre initiated the study of parallel versions of this result  for operators \cite{AAGM1}. For two
normed spaces $X$ and $Y$ over the scalar field $\mathbb{K}$ ($\R$ or $\C$), $\mathcal{L}(X,Y)$ denotes the
space of (bounded and linear) operators from $X$ into $Y,$ endowed with the usual operator norm.

We recall the following definition from \cite{AAGM1}.

\begin{definition}
Let $X$ and $Y$ be both either real or complex Banach spaces.  It is said that  the pair $(X, Y)$
{\it has the Bishop-Phelps-Bollob\'as property for operators (BPBp)}, if for any $\varepsilon
>0$ there exists $\eta(\varepsilon)>0$ such that for any $T\in S_{\mathcal{L}(X,Y)}$, if $x \in
S_X$ is such that $\|Tx\|>1-\eta (\varepsilon)$, then there exist an element $u $ in $S_X$ and an
operator $S$ in $S_{\mathcal{L}(X,Y)}$ satisfying  the following conditions
$$
\Vert Su \Vert=1, \sem \Vert u - x \Vert < \varepsilon \sem
\mbox{and} \sem \Vert S-T \Vert <\varepsilon.
$$
\end{definition}

During the last years there are a number of  interesting results where it is shown  versions of
Bishop-Phelps-Bollob{\'a}s Theorem for operators (see for instance \cite{ACGM}, \cite{CGK} and
\cite{Kim}). It is known  that the pair  $(X,Y)$ has the BPBp whenever $X$ and $Y$ are finite
dimensional spaces (see \cite[Proposition 2.4]{AAGM1}). If a Banach space $Y$ has the property
$\beta$ of Lindenstrauss, then $(X,Y)$ has the BPBp for every Banach space $X$ (see \cite[Theorem
2.2]{AAGM1}). In the case when $X=\ell_1$ a~characterization of the Banach spaces $Y$ such  that
the pair $(\ell_1,Y)$ has the BPBp was given in \cite[Theorem 4.1]{AAGM1}.

It should be pointed out that very little is known about the stability under direct sums of the property that
a pair of Banach spaces $(X,Y) $ has the Bishop-Phelps-Bollob{\'a}s property for operators. In order to state
some results of this kind  we recall  the following notion used in \cite{AcArGaPa}. Given two Banach spaces
$X$ and $Y$ (both real or complex), we say that $Y$ has {\it property $\mathcal{P}_X$} if the pair $(X, Y)$
has the BPBp for operators.

It  was shown in \cite{ACKLM} that  the pairs $\bigl(X, \bigl(\oplus \sum
_{n=1}^\infty Y_n\bigr)_{c_0}\bigr)$ and $\bigl(X, \bigl( \oplus \sum _{n=1}^\infty
Y_n\bigr) _ {\ell_\infty}\bigr)$ satisfy the Bishop-Phelps-Bollob{\'a}s property for
operators whenever all pairs $(X,Y_n)$ have the Bishop-Phelps-Bollob{\'a}s property for
operators \com{uniformly}. In general the analogous stability result does not hold
for every  Banach sequence lattice $E$ instead of $c_0$. For instance, the subset
of norm attaining operators from any Banach space $X$ into $\ell_p$ ($1 \le p <
\infty$) is not dense in the space of operators from $X$ into $\ell_p$ (\cite{Go,
Ac-l}) for every Banach space $X$. Indeed it is a longstanding open question if for
every (real) Banach space $X$, the subset of norm attaining operators from $X$ into
the euclidean space $\R^2$ is dense in the corresponding space of operators.
However, it is also known that $\mathcal{P}_{\ell_1}$ is stable under finite
$\ell_p$-sums for $1 \le p \le \infty $ (see \cite[Corollary 2.8]{AcArGaPa}).

In this paper we provide two nontrivial extensions of  the above stability  results.  On one hand we prove
that the property $\mathcal{P}_{\ell_1}$ is stable under absolute summands (Theorem \ref{th-estable}). This
extends the above  mentioned result for  finite $\ell_p$-sums. We also prove under mild additional
assumptions, that the property $\mathcal{P}_{\ell_1}$ is stable under  $E$-sums, being $E$ a uniformly
monotone Banach sequence lattice (Theorem \ref{th-AHSP-sequence}).
As a consequence we deduce, for instance,  that if $\{X_k:\, k \in \N \}$ is a sequence of spaces such that
$X_k$ is either some $C(K)$ or $L_1 (\mu)$ or a Hilbert space, then the pair $\bigl( \ell_1, \bigl( \sum
_{k=1}^\infty X_k \bigr) _{\ell_p} \bigr) $ has the BPBp for operators (Corollary \ref{cor-sums}).

On the other hand, in case that the range   is a Hilbert space, we also prove some optimal stability result
of BPBp under $\ell_1$-sums on the domain (Proposition \ref{prop-sum-1-BPBP}). This  result  extends
\cite[Proposition 9]{KLM-2015}, where the authors show the  above result for the $\ell_1$-sum of copies of
the same space.

As we already mentioned  there is a~characterization of the Banach spaces $Y$ such that the pair
$(\ell_1, Y)$ has the Bishop-Phelps-Bollob{\'a}s property for operators \cite{AAGM1}. The property on
$Y$ equivalent to the previous fact was called the AHSp (Approximate  hyperplane series
property).

We will need the following definition, where  in what follows by a convex series we mean a series
$\sum \alpha_n$, where $0 \le \alpha_n\le 1$ for each $n \in \N$ and $\sum _{n=1}^\infty \alpha
_n =1$.

\begin{definition}
\label{def-AHSP}
 A Banach space $X$  has the {\it approximate hyperplane series property} (AHSp) if and only if
for every $0 < \varepsilon < 1$ there exists  $0 < \eta < \varepsilon $ such that for every
sequence $\{x_n\} $ in $S_X$ and every convex series $\sum  \alpha _n$ with
$$
\biggl \Vert \sum_{n=1}^\infty \alpha_n x_n \biggr \Vert  > 1-\eta ,
$$
there exist a subset $A\subset \N$ and a subset $\{z_k : k \in A\}\subset S_X$    satisfying
\begin{enumerate}
\item
 $ \ \sum_{k\in A}\alpha_k>1- \varepsilon, $
\item
$  \Vert z_k-x_k  \Vert  <\varepsilon \ \ \text{\ for \ all } k\in A$  \ \ {\text and}
\item
there is $x^\ast \in S_{X^\ast}$ such that $x^\ast(z_k)=1$ for every  $k\in A.$
\end{enumerate}
\end{definition}

We will use the following characterization of the AHSp (see \cite[Proposition 1.2]{AcArGaPa}.)

\begin{proposition}
\label{char-AHSP}
Let $X$ be a Banach space. The following conditions are equivalent{\rm:}
\begin{itemize}
\item[{(a)}] $X$ has the AHSp.
\item[{(b)}] For every $0 < \varepsilon < 1$ there exist $\gamma_X\left(\varepsilon\right)>0$ and $\eta_X
(\varepsilon )>0$ with $\lim_{\varepsilon \to 0}\gamma_X (\varepsilon )=0$ such that for every sequence $\{
x_n \}$ in $ B_X$ and every convex series $\sum_{n} \alpha_n$ with $ \displaystyle{\biggl \Vert
\sum_{k=1}^{\infty}\alpha_kx_k \biggr \Vert > 1-\eta_X (\varepsilon ),} $ there are a subset $A \subseteq \N
$ with $\sum_{k\in A}\alpha_k
>1-\gamma_X (\varepsilon )$, an element $x^* \in  S_{X^*}$, and
$ \{ z_k: k \in A\} \subseteq \bigl ( x^*\bigr)^{-1} (1) \cap B_X$ such that $\Vert z_k-x_k \Vert
<\varepsilon$ for all $k\in A.$
\item[{(c)}] For every $0< \varepsilon < 1$ there exists $0 <
\eta < \varepsilon$ such that for any sequence $\{ x_n \}$ in $ B_X$ and every convex series $\sum_{n}
\alpha_n$ with $ \displaystyle{\biggl \Vert \sum_{k=1}^{\infty}\alpha_kx_k \biggr \Vert > 1-\eta,}$ there are
a subset $A \subset \N $ with $\sum_{k\in A}\alpha_k
>1-\varepsilon $, an element $x^* \in  S_{X^*}$, and
$ \{ z_k: k \in A\} \subseteq \bigl ( x^*\bigr)^{-1} (1) \cap B_X$ such that $\Vert z_k-x_k \Vert
<\varepsilon$ for all $k\in A.$
\item[{(d)}]  The same statement holds as in $(c)$ but for every
sequence $\{ x_n\}$ in $S_X$.
\end{itemize}
\end{proposition}

\section{The main results}

In the section we study the Bishop-Phelps-Bollob\'{a}s property for operators between special
types of Banach spaces. In particular we are interested in stability of  this property  when the
domain is an $\ell_1$ sum of Banach spaces.  Throughout the paper we consider either real or
complex Banach spaces.

We will need the following lemma (see \cite[Lemma 3.3]{AAGM1}).

\begin{lemma}
\label{elemental} Let  $\{c_n\}$ be a sequence  of complex numbers
with $\vert c_n \vert \le 1$ for each $n$ and let $\eta
> 0$ be such that there is some sequence $\{ \alpha _n \}$ of nonnegative real
numbers satisfying $\sum_{n=1}^\infty \alpha_n \le 1$ and \ $ \rea\sum_{
n=1}^{\infty} \alpha _n c_n > 1 - \eta $. Then for every $ 0 < r < 1,$ the set $A:=
\{ i\in {\mathbb{N}}: \rea c_i
> r\}$, satisfies the estimate
$$
\sum _{ i \in A} \alpha _i > 1 - \frac{\eta}{1-r}.
$$
\end{lemma}

We also need the following technical lemma. For the sake of
completeness we include a~proof.

\begin{lemma}
\label{iso-hilb}
Let $H$ be a~real or complex  Hilbert space and assume that $u, v \in S_H$. Then
there is a~surjective linear isometry $\Phi $ on $H$ such that $\Phi(u)=v$ and
$\Vert \Phi - I \Vert  = \Vert u - v \Vert$.
\end{lemma}

\begin{proof}
The result is  obvious in the case $\text{dim}\,H =1$. Assume  that $\text{dim}\,H \ge 2$.  Thus
there is an element $v^\perp \in S_H$ orthogonal to $v$ and such that $[u,v] \subset [v,
v^\perp]$, where $[x,y]$ is the linear span of the vectors $x $ and $y$ in $H$. Let $u_1, u_2 \in
\mathbb{K}$ such that $u= u_1 v + u_2 v^\bot$ and write $u^\bot =- \overline{u_2} v +
\overline{u_1} v^\bot$. It is clearly satisfied that
$$
1= \Vert u \Vert ^2 =  \vert u_1 \vert ^2 +  \vert u_2 \vert ^2
\seg \text{and} \seg \langle u, u^\perp \rangle = 0.
$$

Let $M$ be a subspace of $H$ orthogonal to $[v, v^\perp]=[u, u^\perp]$ and such
that $H=[u, u^\perp] \oplus M$.  Define the mapping $\Phi: H \llll H$ given by
$$
\Phi (zu+wu^\perp +m)=  zv+wv^\perp +m, \seg \forall (z,w) \in \mathbb{K}^2, m \in
M,
$$
which is a surjective linear isometry on $H$. It clearly  satisfies $\Phi (u)=v$
and $\Phi (u^\perp)=v^\perp$.

Clearly  $(\Phi -I) (u)= v-u,  (\Phi -I) (u^\perp)=
v^\perp-u^\perp$ and $\Vert u-v\Vert = \Vert u^\perp - v^\perp
\Vert$. Also we have that
$$
\langle v-u, v^\perp-u^\perp \rangle =  - \bigl( \langle v,
u^\perp \rangle + \langle u, v^\perp \rangle\bigr) = 0.
$$
Hence $\Phi -I$ restricted to  $[u,u^\perp]$  is a multiple of a
linear isometry from this subspace into itself. As a~consequence
$\Vert \Phi - I \Vert = \Vert u-v \Vert$.
\end{proof}

 The next result uses the argument outlined in \cite[Proposition 9]{KLM-2015} in
the case that the domain is the $\ell_1$-sum of one space.

\begin{proposition}
\label{prop-sum-1-BPBP}
 Assume that  $\{X_i: i \in I\}$ is a
family of Banach spaces, $H$ is a Hilbert space such that the pair $(X_i, H)$ has the BPBp for operators for
every $i \in I$ and with the same function $\eta$. Then the pair $\bigl( \bigl( \oplus \sum _{i \in I} X_i
\bigr) _{\ell_1}, H \bigr) $ has the BPBp.
\end{proposition}

\begin{proof} We write $Z=\bigl( \oplus \sum _{i \in I} X_i \bigr) _{\ell_1}$. Given $0 < \varepsilon < 1$, we choose
positive real numbers  $r,s$ and $t$ such that
\begin{equation}
\label{choice-epsilons}
r  < \dfrac{ \varepsilon}{4}, \sem   s < \min \Bigl\{ \dfrac{ \varepsilon}{4}, \dfrac{ \delta_H (r)}{3}
\Bigr\} \sem \text{and} \sem t <  \min \Bigl\{ \dfrac{ \varepsilon}{4},  \eta (s), \dfrac{ \delta_H (r)}{3}
\Bigr\},
\end{equation}
where $\delta_H $ is the modulus of convexity of $H$.

Assume that $z_0 = \{ z_0(i)\} \in S_Z$ and  $T \in S_{\mathcal{L}
(Z,H)}$ satisfies $\Vert T(z_0) \Vert > 1 -t^2$.  For every $i \in
I$, we denote by $T_i$ the restriction of $T$ to $X_i$, that is
embedded in $Z$ in a natural way. Assume that $y^* \in S_{ H^*}$
satisfies that $\Rea y^* (T (z_0)) = \Vert T(z_0) \Vert > 1 - t^2
$.

Denote by $B=\{ i \in I : \Rea y^* (T_i (z_0(i))) > (1- t) \Vert z_0(i) \Vert \}$.
We clearly have that
\begin{eqnarray*}
1-t^2& < &  \Rea y^* (T(z_0)) = \sum _{ i \in I}  \Rea y^* (T_i(z_0(i)))  \\
& = &  \sum _{ i \in B}  \Rea y^* (T_i(z_0(i)))  + \sum _{ i \in I\backslash B }  \Rea y^* (T_i(z_0(i)))  \\
&\le&  \sum _{ i \in B}  \Vert z_0(i) \Vert + \sum _{ i \in I\backslash B }  (1-t) \Vert z_0(i) \Vert   \\
&=&  1-t \sum _{ i \in I\backslash B }   \Vert z_0(i) \Vert .
\end{eqnarray*}
Hence

\begin{equation}
\label{sum-B} \sum _{ i  \in I\backslash B} \Vert z_0 (i) \Vert
\le t.
\end{equation}
By assumption, for every $i \in B$ there is an operator $S_i \in
S_{\mathcal{L} (X_i, H)}$ and an element $x_i \in S_{X_i}$ such
that
\begin{equation}
\label{Xi-BPBP} \Bigl \Vert S_i - \frac{T_i}{ \Vert T_i \Vert }
\Bigr \Vert < s, \sem \Bigl \Vert x_i - \frac{z_0(i)}{ \Vert z_0(i) \Vert } \Bigr \Vert < s \sem \text{and}
\sem \Vert S_i (x_i) \Vert =1, \seg  \forall i \in B.
\end{equation}

It follows by \eqref{Xi-BPBP} that for every $i, j\in B$ we have that
\begin{eqnarray*}
\Vert S_i (x_i) +  S_j (x_j) \Vert& \ge &  \biggl\Vert  \frac{S_i (z_0(i))}{\Vert
z_0(i)\Vert } + \frac{S_j (z_0(j))}{\Vert z_0(j)\Vert } \biggr \Vert - 2 s\\
& \ge &  \biggl\Vert  \frac{T_i (z_0(i))}{\Vert T_i \Vert \Vert z_0(i) \Vert  } +
 \frac{T_j (z_0(j))}{\Vert T_j \Vert \Vert z_0(j) \Vert  } \biggr \Vert - 4 s\\
&\ge& 2 \bigl( 1- t \bigr) -4s\\
&>& 2 \bigl( 1 - \delta_H (r)\bigr).
\end{eqnarray*}
As a consequence $\Vert S_i (x_i) - S_j (x_j) \Vert \le r $ for each $i,j \in B$.

Since $B \ne \varnothing$, we choose some element $i_0 \in B$ and define $y_0= S_{i_0} (x_{i_0})$. By Lemma
\ref{iso-hilb}, for every $i \in B$, there is a linear surjective isometry $\Phi_i : H \to H$  such that
$\Phi_i (S_i (x_i))= y_0$ and $\Vert \Phi_i - I \Vert  =  \Vert S_i (x_i) - y_0 \Vert \le r $.

We define an operator $R= \{R_i\}_{i \in I}\in \mathcal{L}(Z,H)$
by
$$
R_i = \Phi_i \circ S_i, \quad\, \forall i \in B \quad \text{and} \quad\, R_i = T_i, \quad\, \forall i \in I
\backslash B.
$$

Clearly that $R$ is in the unit ball of $\mathcal{L}(Z,H)$ and it satisfies
\begin{eqnarray*}
\Vert R-T \Vert &=&  \sup \{ \Vert R_i- T_i \Vert : i \in B \}\\
&\le&\sup \{ \Vert \Phi_i - I \Vert : i \in B \} +   \sup \{ \Vert S_i - T_i \Vert : i \in B \} \\
&\le & r + \sup \biggl \{ \biggl\Vert S_i - \frac{T_i }{ \Vert T_i \Vert} \biggr
\Vert
  : i \in B \biggr \}+
\sup \biggl \{ \biggl \Vert\frac{T_i }{ \Vert T_i \Vert} - T_i   \biggr \Vert : i \in B \biggr\}\\
&\le & r + s +  \sup \bigl \{ \bigl  \vert 1-  \Vert T_i \Vert  \bigr \vert : i \in B \bigr\} \\
 &\le & r + s + t <    \varepsilon .
\end{eqnarray*}
Let $P_B$ be  the natural projection on the subspace of elements in $Z$ whose support is contained in $B$.

Now observe that $x_0$ given by
$$
x_0 (i)=  \left\{%
\begin{array}{ll}
\dfrac { \Vert z_0(i) \Vert x_i }{\Vert P_B (z_0) \Vert}, & \hbox{if} \ i \in B \\
& \\
0 & \hbox{if} \ i \in I \backslash B \\
\end{array}%
\right.
$$
belongs to $S_Z$ and also satisfies
\begin{eqnarray*}
\Vert x_0 -z_0 \Vert &\le& \bigl \Vert x_0  -  \Vert P_B (z_0) \Vert x_0 \bigr \Vert + \bigl
\Vert \Vert P_B (z_0) \Vert x_0 - z_0 \chi_B \bigr \Vert +
\bigl \Vert   z_0 \chi_{I \backslash B} \bigr \Vert\\
&\le& \bigl \vert 1- \Vert P_B (z_0) \Vert   \, \bigr\vert +
\sum _{i \in B} \bigl \Vert \Vert z_0(i) \Vert x_i - z_0 (i) \bigr \Vert   +
\bigl \Vert   z_0 \chi_{I \backslash B} \bigr \Vert\\
&\le & 2 \bigl \Vert   z_0 \chi_{I \backslash B} \bigr \Vert
+ s\,\sum _{i \in B}
\Vert z_0(i) \Vert \sem \text{(by \eqref{Xi-BPBP})} \\
&\le &  2t+s \sem \text{(by \eqref{sum-B})}  \\
&<  &  \varepsilon.
\end{eqnarray*}
It remains to check that  $R$ attains its norm at $x_0$. Indeed,
\begin{eqnarray*}
\Vert R(x_0) \Vert &= & \frac{1}{ \Vert P_B(z_0) \Vert }  \Bigl \Vert \sum _{ i \in
B} \Vert
z_0(i) \Vert  R_i (x_i) \Bigr \Vert \\
& = & \frac{1}{ \Vert P_B(z_0) \Vert }  \Bigl \Vert \sum _{ i \in B} \Vert
z_0(i) \Vert  \Phi_i (S_i (x_i)) \Bigr \Vert \\
& = & \frac{1}{  \Vert P_B(z_0) \Vert }  \Bigl \Vert \sum _{ i \in B} \Vert
z_0(i) \Vert  y_0 \Bigr \Vert = 1.\\
\end{eqnarray*}
Hence $R \in S_{ \mathcal{L}(Z,H)}$ and $\Vert R (x_0) \Vert =1$.  This completes
the proof that the  pair  $(Z, H)$ has the BPBp.
\end{proof}

Let us note that it follows from \cite[Theorem 2.1]{ACKLM} that
$(X_i,H)$ has the BPBp for every  $i\in I$ with the same function
$\eta$ provided that  $ \bigl(\bigl( \oplus \sum _{i \in I} X_i
\bigr) _{\ell_1}, H \bigr)$ has the BPBp. This shows that the
assumption in Proposition \ref{prop-sum-1-BPBP} is a necessary
condition.

Now we prove stability results of  the Bishop-Phelps-Bollob\'{a}s
property for operators when the domain is $\ell_1$.

As we already mentioned it was proved that the pair $(\ell_1,Y)$ has the BPBp for operators if,
and only if, $Y$ has the approximate hyperplane series  property (see \cite[Theorem 4.1]{AAGM1}).
Since the AHSp is an isometric property, if a space is the (topological) direct sum of two
subspaces with the AHSp, in general it does not have the AHSp. However, we will prove that this
property is stable under  sums involving an absolute  (or monotone) norm. First we recall this
notion.

\begin{definition}
\label{def-nor-absoluta}
Let  $X$ and $Y$ be  Banach spaces, and $Z=X\oplus Y$, a norm $\Vert\cdot\Vert_f$ in $Z$ is said
to be {\it absolute} if there is a~function $f: \R^+_{0} \times \R^+_{0} \llll \R^+_{0} $ such
that
\begin{equation}
\label{def-absolute-norm}
\Vert x+ y \Vert_f=f(\Vert x\Vert ,\Vert y\Vert),\ \ \ \forall x\in X,\ y\in Y\ .
\end{equation}
The absolute norm is {\it normalized} if $f(1,0)=1=f(0,1)$.
\end{definition}

It is immediate to check that in   case that the  equality \eqref{def-absolute-norm} gives a norm
in $Z$, the function $f$ can be extended to a norm $\vert \cdot \vert$ on $\R^2$ satisfying
$\vert (r,s) \vert =  f ( \vert r \vert, \vert s \vert)$ for every pair of real numbers $(r,s)$.

We also recall that the norm $\vert \cdot \vert_{f}$ is   absolute on $\R^2 $ if, and only if, it
satisfies
$$
\vert  r \vert  \le \vert  s \vert , \  \vert t  \vert  \le  \vert  u  \vert \ \
\Rightarrow \ \ f(r, t ) \le f(s, u )
$$
(see  for instance \cite[Lemma 21.2]{BoDu}).

Clearly the usual $\ell_p$-norm of the sum of two Banach spaces is an absolute norm for every
$1\leq p\leq \infty$.

Next result is a far reaching extension of   Proposition 2.1, Theorems 2.3 and 2.6 in
\cite{AcArGaPa}, where  the $\ell_p$-norm on $\R^2$ for $1 \le p < \infty$  is considered. Part
of the essential idea of the argument we will use is contained there, however  our proof is
simpler.

The following technical lemma will be useful in the proof of the main result.

\vskip5mm

\begin{lemma}
\label{lemma-fact}
Let $\vert \cdot  \vert$ be an absolute and normalized norm on $\R^2$. For every $\varepsilon> 0$
there is $\delta > 0$ satisfying the following conditions{\rm:}
$$
(r,s) \in \R^2,  \ \ \vert (r,s) \vert =1, \   s> 1 - \delta \ \ \Rightarrow \ \ \exists t \in \R: \vert
(t,1) \vert =1 \ \ \text{\rm and} \ \ \vert t-r \vert < \varepsilon
$$
and
$$
(r,s) \in \R^2,   \ \vert (r,s) \vert =1 , \   r> 1 - \delta \ \ \Rightarrow \ \ \exists t \in \R: \vert
(1,t) \vert =1 \ \ \text{\rm and} \ \ \vert t-s \vert < \varepsilon.
$$
\end{lemma}

\begin{proof}
Of course it suffices to  check only the first assertion. Assume that it is not true. Hence there
is some $\varepsilon_0> 0$ such that
$$
\forall \delta > 0 \ \exists (r_\delta, s_\delta ) \in S_{ (\R^2, \vert \cdot \vert )}, \
s_\delta
> 1 - \delta \sep \text{and} \sep t \in \R \ \text{with} \ \vert (t,1) \vert =1\sep \Rightarrow
 \vert t- r_\delta \vert \ge \varepsilon_0.
$$
We choose any sequence $\{ \delta _n\}$ of positive real numbers converging to $0$. By assumption
there is a sequence  $\{(r_n, s_n)\}$ in $S_{(\R^2, \vert \cdot \vert)}$ satisfying for each
$n\in \N$ that
\begin{equation}
\label{far-away}
s_n > 1 - \delta _n \seg \text{and} \seg \vert t- r_n \vert \ge \varepsilon_0 \sep \forall t \in \R \sep
\text{with} \sep \vert (t,1) \vert =1.
\end{equation}
By passing to a subsequence, we may assume that  $(r_n, s_n ) \to (r,s).$ Since $\vert (0,1)\vert
=1 $ and the norm is absolute  on $\R^2$ it is satisfied
$$
s= \vert (0,s) \vert \le  \vert (r,s) \vert =1.
$$
Since $s_n > 1 - \delta _n$ for each $n$ we also have $s \ge 1$. So $s=1$. So $\vert (r, 1) \vert
=1$. We also know that  $r_n \to r$,  hence  $(r_n , s_n ) \to (r,1)$ and this contradicts
condition
\eqref{far-away}.
\end{proof}

\bigskip

\begin{theorem}
\label{th-estable}
Assume that  $\vert \cdot \vert$ is an absolute  and normalized norm on $\R^2$. Let $X$ be
a~$($real or complex$)$ Banach space that can be decomposed as $X= M \oplus N$ for certain
subspaces $M$ and $N$ and such that
$$
\Vert (m,n) \Vert = \vert ( \Vert m \Vert , \Vert n \Vert ) \vert,  \ \ \ \ \forall  m \in M,\  n
\in N.
$$
Then $X$ has the AHSp if, and only if,  both $M$ and $N$ has the AHSp. In such case, both
subspaces satisfy Definition \ref{def-AHSP} with the same function $\eta $.
\end{theorem}

\begin{proof}
We can clearly assume that both $M$ and $N$ are non-trivial. Let $P$ and $Q$ be the natural
projections from $X$ onto $M$ and $N$, respectively.

First we check the necessary condition.  So assume that $X$ has the AHSp and we show that $M$
also has the AHSp. Let us fix $0 < \varepsilon < 1$ and let $\eta _0 $  be the positive number
satisfying Definition \ref{def-AHSP} for the space $X$  and $\varepsilon/2$.

Assume that  $\sum_{k=1}^{\infty}\alpha_k m_k$ is a convex series with $\{ m_k : k \in A \}
\subset S_M$ satisfying
$$
\Bigl \Vert \sum_{k=1}^{\infty}\alpha_k m_k \Bigr \Vert > 1- \eta_0.
$$
By  the assumption  there are $A \subset \N$ and  $\{ x_k : k \in \N \} \subset S_X$ such that
$$
\sum _{ k \in A} \alpha _k >  1- \frac{\varepsilon}{2} > 0, \  \  \ \ \Vert x_k - m_k \Vert <
\frac{\varepsilon}{2}, \ \ \  \forall k \in A \seg \text{and} \seg \text{co} \bigl\{ x_k : k \in A \bigr\}
\subset S_X.
$$
So $A \ne \varnothing$.

Since the norm $\vert \cdot  \vert $  on $\R^2$ is an absolute norm it is satisfied
\begin{equation}
\label{Pxk-mk}
 \Vert P(x_k) - m_k
\Vert =  \Vert P(x_k - m_k ) \Vert  \le \Vert x_k - m_k \Vert < \frac{\varepsilon}{2},
\end{equation}
and
$$
\Vert Q(x_k) \Vert \le   \Vert x_k - m_k  \Vert   < \frac{\varepsilon}{2}.
$$
Hence we have that
\begin{equation}
\label{Pxk-Qxk} \Vert P(x_k)  \Vert > 1- \frac{\varepsilon}{2}
\seg \text{and} \seg  \Vert Q(x_k) \Vert  < \frac{\varepsilon}{2}, \seg  \forall k \in A.
\end{equation}
On the other hand, since $\text{co} \bigl\{ x_k : k \in A \bigr\} \subset S_X$ there is $x^* \in
S_{X^*}$ that can be decomposed as $x^* =m^* + n^*$, for some $m^* \in M^*$ and $n^* \in N^*$ and
such that for each $k \in A$ it is satisfied
\begin{eqnarray}
\label{norma-funcional}
1 &=&  \rea x^*\left(x_k\right) \notag\\
&=&  \rea  m^*\bigl( P(x_k) \bigr)  +  \rea n^ * \bigl( Q(x_k) \bigr) \notag \\
&\leq&  \bigl \Vert  m^* \bigr \Vert \  \Vert    P(x_k)  \Vert + \bigl \Vert  n^* \bigr \Vert
\ \Vert    Q(x_k)  \Vert \\
& =&  \bigl ( \bigl \Vert  m^* \bigr \Vert, \bigl \Vert  n^* \bigr
\Vert \bigr) \Bigl ( \bigl
\Vert  P(x_k)\Vert, \Vert  Q(x_k) \Vert   \bigr) \notag \\
&\le& \Vert x^* \Vert  \Vert x_k \Vert = 1. \notag
\end{eqnarray}

As a consequence, we obtain that
\begin{equation}
\label{m*Pxk} m^* \bigl( P(x_k) \bigr) = \Vert m^* \Vert \ \Vert
P(x_k) \Vert,   \seg   \forall k \in A.
\end{equation}

Let us fix $k \in A$. If $m^*=0$, in view of \eqref{norma-funcional} we obtain that $\Vert Q(x_k
) \Vert =1$, which contradicts \eqref{Pxk-Qxk}. By using again \eqref{Pxk-Qxk}  we also know that
$P(x_k) \ne 0$, so we can write $u_k= \frac{ P(x_k)} {\Vert P(x_k) \Vert }$. By \eqref{m*Pxk}
we obtain that
$$
\frac{m^*}{ \Vert m^* \Vert }(u_k)=1 \seg  \forall k \in A
$$
and clearly  $\frac{m^*}{ \Vert m^* \Vert } \in S_{M^*} \subset S_{X^*}$.

For $k \in A$ we also have
\begin{eqnarray*}
\label{uk-mk-nec}
 \Vert u_k - m_k \Vert & \le \Bigl \Vert \frac{ P(x_k)}{\Vert P(x_k) \Vert
} - P(x_k) \Bigr \Vert + \Vert P(x_k) -  m_k \Vert  \\
 & \le \bigl \vert 1- \Vert P(x_k) \Vert  \bigr \vert  + \Vert P(x_k) - m_k \Vert \\
 & < \varepsilon \seg \text{(by \eqref{Pxk-Qxk} and \eqref{Pxk-mk})}.
\end{eqnarray*}
We checked that $M$ has the AHSp.

Conversely, assume that $M $ and $N$ have the  AHSp. We will prove that $X$ also has the AHSp.
Let  $ \varepsilon $ be a real number with $0 < \varepsilon < 1$. In view of Lemma
\ref{lemma-fact} there is  $0 < \delta < 1 $ satisfying the following conditions
\begin{equation}
\label{close-var-1}
 (a,b) \in S_{ (\R^2, \vert \cdot \vert)},  \sep b> 1 - \delta \ \ \Rightarrow \ \ \exists c \in \R: \vert
(c,1) \vert =1 \ \ \text{and} \ \ \vert a-c \vert < \frac{\varepsilon}{5}
\end{equation}
 and
\begin{equation}
\label{close-var-2}
(a,b) \in S_{ (\R^2, \vert \cdot  \vert)}, \seg a> 1 - \delta \ \ \Rightarrow \ \ \exists c \in
\R: \vert (1,c) \vert =1 \ \ \text{and} \ \ \vert b-c \vert < \frac{\varepsilon}{5}.
\end{equation}

Let us choose $0 < \varepsilon_1  < \frac{\varepsilon}{8}$. Assume that the pair $(
\varepsilon_1, \eta _1)$ satisfy  condition (c) in Proposition \ref{char-AHSP}  for both $M$ and
$N$. We also fix  real numbers $r, s $ and $\varepsilon_0$ such that
\begin{equation}
\label{parameters}
0<  s< \min \Bigl\{ \frac{\delta}{2},  \frac{\eta _1}{2} \Bigr\}, \sem 0< r < \min \Bigl\{\frac{\delta}{2},
s^2 \eta _1  \Bigr\} \sem \text{and} \sem 0<  \varepsilon_0 < \frac{r \varepsilon}{8}.
\end{equation}

By \cite[Proposition 3.5]{AAGM1} finite-dimensional spaces have the AHSp. So for every
$\varepsilon_0 > 0$ there is $ 0 < \eta _0 <   \varepsilon_0$ satisfying condition (d) in
Proposition \ref{char-AHSP} for $\R^2 $ endowed with the norm $\vert \cdot \vert$.

Let  $\{x_k\} $ be a sequence in $S_X$ and $\sum \alpha _k$ be a~convex series such that
\linebreak[4] $\displaystyle{\biggl \Vert \sum_{ k=1}^\infty \alpha _k x_k \biggr\Vert > 1 - \eta
_0}$. Hence we have
\begin{align*}
\label{1-norm-negat-part-1}
 1- \eta _0  & < \biggl \Vert \sum _{k=1}^\infty \alpha _k x_k \biggr \Vert  \notag \\
& = \biggl \Vert \sum _{k=1}^\infty \alpha _k \bigl( P(x_k) + Q(x_k) \bigr) \biggr \Vert  \notag \\
&=  \biggl \vert \biggl( \Bigl\Vert  \sum _{k=1}^\infty \alpha _k  P(x_k) \Bigr
\Vert, \Bigl\Vert
\sum _{k=1}^\infty \alpha _k  Q(x_k) \Bigr \Vert \biggr) \biggr \vert  \notag \\
& \le  \biggl \vert \biggl( \sum _{k=1}^\infty \alpha _k   \bigl \Vert  P(x_k)
\bigr\Vert, \sum _{k=1}^\infty \alpha _k   \bigl\Vert  Q(x_k) \bigr\Vert \biggr)
\biggr \vert  \notag \\
& = \biggl \vert \sum _{k=1}^\infty \alpha _k  \bigl( \bigl\Vert P(x_k) \bigr
\Vert, \bigl\Vert Q(x_k)  \bigr\Vert \bigr) \biggr \vert \notag.
\end{align*}

Since $(\R^2, \vert \cdot \vert) $ has the AHSp, it follows that for the convex series
\linebreak[4]
$ \sum _{k=1}^\infty \alpha _k  \bigl( \bigl\Vert P(x_k) \bigr \Vert, \bigl\Vert Q(x_k)
\bigr\Vert \bigr)$, there are a subset $A \subset \N$, $\{ (r_k,s_k): k \in A \} \subset
S_{\R^2}$ and $(\alpha, \beta ) \in S_{ (\R^2)^{*}}$ satisfying
\begin{equation}
\label{sum-A}
\sum _{k \in A} \alpha _k > 1 - \varepsilon_0, \seg  r_k, s_k \ge 0, \seg \alpha r_k + \beta s_k =1,\seg
\forall k \in A,
\end{equation}
and
\begin{equation}
\label{P-Q-r-s}
\bigl \vert \Vert P(x_k) \Vert - r_k \bigr \vert < \varepsilon_0,\seg
\bigl \vert \Vert Q(x_k) \Vert - s _k \bigr \vert < \varepsilon_0, \ \ \  \forall   k \in A.
\end{equation}

It is clearly satisfied that
\begin{align}
\label{sum-A-xk}
\nonumber \biggl \Vert \sum _{k \in A } \alpha _k x_k  \biggr \Vert  & \ge     \biggl \Vert \sum
_{k=1}^{\infty} \alpha _k x_k \biggr \Vert
-  \biggl \Vert \sum _{ k \in \N \backslash A}  \alpha _k x_k   \biggr \Vert \\
& \ge   \biggl \Vert \sum _{k=1}^\infty  \alpha _k  x_k   \biggr \Vert - \sum _{k \in \N
\backslash A} \alpha _k \\
\nonumber  & >   1-\eta _0 -\varepsilon_0         \ \ \ \text{(by \eqref{sum-A})}\\
\nonumber  & >    1- 2\varepsilon_0.
\end{align}

Now fix arbitrary elements $m_{0} \in S_M$ and $n_0 \in S_N$ and define the following elements:
$$
m_k := \left\{
\begin{array}{ll}
\frac{r_kP\left(x_k\right)} {\left\| P\left(x_k\right) \right\|} & \text{\
if\ } k \in A \text {\ and\ } P\left(x_k\right) \ne 0 \\
r_km_0 & \text{\ if\ } k \in A \text {\ and\ } P\left(x_k\right) =0 \\
\end{array}
\right.
$$
and
$$
n_k := \left\{
\begin{array}{ll}
\frac{s_kQ\left(x_k\right)} {\left\| Q\left(x_k\right) \right\|} & \text{\ if\ } k \in A \text
 {\ and\ } Q\left(x_k\right) \ne 0 \\
s_kn_0 & \text{\ if\ } k \in A \text {\ and\ } Q\left(x_k\right) =0.\\
\end{array}
\right.
$$

Next we write  $y_k := m_k + n_k$ for all $k \in A$. Since $\vert (r_k, s_k )\vert=1$ for every
$k \in A$, it is clear that $\left\{y_k:k \in A \right\} \subset S_X$ and in view of
\eqref{P-Q-r-s} we obtain
\begin{equation}
\label{y-x} \left\| y_k - x_k \right\| \le  \vert r_k -  \Vert P(x_k)
\Vert\vert + \vert s_k - \Vert Q(x_k) \Vert  \vert < 2 \varepsilon_0, \seg  \forall k \in A.
\end{equation}
 By the previous inequality and bearing in mind \eqref{sum-A-xk}  we have
$$
\biggl \Vert \sum_{k\in A} \alpha_k y_k \biggr \Vert > \biggl \Vert \sum_{k\in A}   \alpha _k x_k \biggr
\Vert-2\varepsilon_0 >  1- 4\varepsilon_0.
$$
In view of Hahn-Banach theorem there is a functional $x^* \in S_{X^*}$ such that
$$
\rea x^* \biggl( \sum _{ k \in A} \alpha _k y_k \biggr) =  \biggl \Vert  \sum _{ k \in A} \alpha _k y_k
\biggr \Vert > 1 -  4 \varepsilon_0.
$$
Now we  define $B= \bigl\{ k \in A : \rea x^*(y_k) > 1 -  r \bigr\}$. In view of Lemma \ref{elemental} we
have that
\begin{equation}
\label{sum-B-alfa}
\sum _{ k \in B }  \alpha _k > 1 - \frac{4 \varepsilon_0}{ r} > 0.
\end{equation}

If we decompose $x^* = m^* + n^*$, for  each  $k \in B$ we have that
\begin{align}
\label{estado}
1- r&
 < \rea x^* (y_k) =  \rea \bigl(  m^* (m_k) + n^* (n_k) \bigr)  \notag \\
&
\le    \Vert m^* \Vert \Vert m_k \Vert + \rea n^* (n_k)   \\
& \le    \Vert m^* \Vert \Vert m_k \Vert + \Vert n^* \Vert \Vert
n_k \Vert \notag  \le  1.
\notag
\end{align}

As a consequence of \eqref{estado}, for each $k \in B$, we also have that
\begin{equation}
\label{m-Px-R}
  \Vert m^* \Vert  r_k =  \Vert m^* \Vert  \Vert m_k \Vert   \le \rea m^* (m_k) + r
\end{equation}
and
\begin{equation}
\label{n-Qx-R}
\Vert n^* \Vert  s_k = \Vert n^* \Vert \Vert n_k
\Vert \le \rea n^* (n_k) + r.
\end{equation}

In order to show the result  we will consider three  cases:

\vskip5mm

 Case 1)  Assume that $ \Vert m^* \Vert  \le s $.
\newline
Since $\Vert n^* \Vert \le \Vert x^* \Vert =1$, in view of \eqref{estado}
 we know that
\begin{equation}
\label{s-k-big} s_k \ge \Vert n^* \Vert s_k \ge 1-r - s> 1 -
\delta , \seg  \forall  k \in B.
\end{equation}
 By using also \eqref{n-Qx-R} we obtain that
$$
\rea n^*  (n_k) \ge 1-2r -s > 1 - \eta _1,  \seg    \forall k \in B.
$$

Since $N$ has the AHSp there are  $C \subset B$,  $\{ v_k: k \in C \} \subset S_N$  and $n_{1}^*
\in S_{N^*}$ such that
\begin{equation}
\label{AHSP-N} \sum _{k \in C} \alpha _k > (1- \varepsilon_1) \sum
_{k \in B} \alpha _k, \sep n_{1}^* (v_k)= 1\sep \text{and} \sep \bigl \Vert  v_k - n_k \bigr \Vert  <
\varepsilon_1, \sem   \forall k \in C.
\end{equation}

By \eqref{s-k-big} we  can use \eqref{close-var-1},  and so  for every $k \in C$ there is $a_k
\in \R$ such that
\begin{equation}
\label{ak-rk}
\vert (a_k, 1) \vert =1, \sem \vert a_k - r_k \vert < \frac{ \varepsilon}{5}.
\end{equation}
So we define the subset $\{ z_k: k \in C\}\subset X$  by
$$
z_k= a_k \frac{m_k}{\Vert m_k \Vert} + v_k \sep \text{if }\sep m_k \ne 0, \seg
z_k= a_k m_0 + v_k \sep \text{if }\sep m_k=0, \seg   \forall k \in C.
$$
Clearly we have that
$$
\Vert z_k\Vert = \vert (a_k, 1)\vert =1, \sem   \forall  k \in C.
$$
By  \eqref{y-x}, \eqref{ak-rk} and \eqref{AHSP-N}  we obtain that
\begin{eqnarray*}
\Vert z_k - x_k \Vert&\le &  \Vert  z_k - y_k
\Vert + \Vert y_k - x_k \Vert\\
&\le&  \vert a_k - r_k\vert + \Vert v_k -  n_k \Vert + 2 \varepsilon_0  \\
&\le & \frac{\varepsilon}{5} + \varepsilon_1 + 2 \varepsilon_0 \\
&< &\varepsilon.
\end{eqnarray*}

We also have that
$$
n_{1}^* (z_k)=  n_{1}^* (v_k)=1, \seg   \forall k \in C.
$$

Finally from \eqref{AHSP-N} and \eqref{sum-B-alfa} we also know that
$$
\sum _{k \in C} \alpha _k > (1- \varepsilon_1)  \sum _{k \in B} \alpha _k > (1- \varepsilon_1) \Bigl(1-
\frac{4\varepsilon_0}{r}\Bigr) > 1- \varepsilon_1 - \frac{4\varepsilon_0}{r} > 1- \varepsilon.
$$
So the proof is finished in this case.

Case 2)  Assume that $ \Vert n^* \Vert  \le s $.
\newline
We can proceed in the same way that in Case 1, but by using that $M$ has the AHSp.

Case 3) Assume that $ \Vert m^* \Vert , \Vert n^* \Vert > s $.
\newline
We define the set $B_1$ given by
$$
B_1= \{ k \in B : r_k \ge s\}.
$$
For  each   element  $k \in B_1$, in view of \eqref{m-Px-R} we have that
$$
\frac{\rea m^* (m_k)}{ \Vert m^* \Vert r_k} \ge 1- \frac{r}{ \Vert
m^* \Vert r_k  } \ge 1- \frac{r}{s^2} > 1 - \eta_{1}.
$$

Since $M$ has the AHSp there is a set $D_1 \subset B_1$,  $\{ u_k: k  \in D_1 \} \subset S_M$
and $m_{1}^* \in S_{M^*}$ such that
\begin{equation}
\label{sum-D1}
\sum_{k\in D_1} \alpha _k \ge (1- \varepsilon_1)  \sum_{k\in B_1} \alpha _k \ge \sum_{k\in B_1} \alpha _k -
\varepsilon_1
\end{equation}
and
\begin{equation}
\label{uk-mk}
\Bigl \Vert u_k - \frac{ m_k}{r_k} \Bigr \Vert < \varepsilon_1, \seg m_{1}^* (u_k)=1, \seg \forall k \in D_1.
\end{equation}

In an analogous way, we  can proceed by defining  the set $C_1=\{k \in B : s_k \ge s\}$ and by
using  that $N$ has the AHSp we obtain that there is a set $F_1 \subset C_1$,   $\{ v_k: k  \in
F_1 \} \subset S_N$  and $n_{1}^* \in S_{N^*}$ such that
\begin{equation}
\label{sum-F1}
\sum_{k\in F_1} \alpha _k \ge (1- \varepsilon_1)  \sum_{k\in C_1} \alpha _k \ge \sum_{k\in C_1} \alpha _k -
\varepsilon_1
\end{equation}
and
\begin{equation}
\label{vk-nk} \Bigl \Vert v_k - \frac{ n_k}{s_k} \Bigr \Vert <
\varepsilon_1, \seg n_{1}^* (v_k)=1, \seg \forall  k \in F_1.
\end{equation}

Let us notice that for $k \in B \backslash B_1$ we have that $r_k \le s$ and since $1=\vert (r_k,
s_k) \vert \le s +  s_k < \frac{1}{2} + s_k $ then $s_k >\frac{1}{2}>  s $. Hence $k \in C_1$.
Hence we checked that
\begin{equation}
\label{B-B1-C1}
 B \backslash B_1 \subset C_1 \seg \text{and \ so } \seg B \backslash C_1 \subset B_1.
\end{equation}

Clearly we have that
\begin{eqnarray}
\label{sum-B1-C1}
\sum _{ k \in B_1 \cap C_1} \alpha _k & \le & \sum _{ k \in D_1 \cap F_1 } \alpha _k +  \sum _{ k \in B_1
\backslash D_1}
 \alpha _k +  \sum _{ k \in C_1 \backslash F_1}  \alpha _k\\
\nonumber &\le&   \sum _{ k \in D_1 \cap F_1 } \alpha _k  + 2 \varepsilon_1  \sem \text{(by \eqref{sum-D1}
and \eqref{sum-F1}).}
\end{eqnarray}

We also obtain
\begin{eqnarray}
\label{sum-B-B1}
\nonumber \sum _{ k \in B \backslash B_1} \alpha _k & = & \sum _{ k \in (B \backslash B_1)
\cap F_1} \alpha _k + \sum _{ k \in B \backslash (B_1 \bigcup F_1)  } \alpha _k\\
 &\le& \sum _{ k \in (B \backslash B_1) \cap F_1}  \alpha _k  +  \sum_{ k \in C_1 \backslash F_1 }
\alpha _k   \sem \text{(by \eqref{B-B1-C1})} \\
\nonumber  & \le & \sum _{ k \in (B \backslash B_1) \cap F_1} \alpha _k + \varepsilon_1   \sem \text{(by
\eqref{sum-F1}).}
\end{eqnarray}

By arguing as above we get
\begin{eqnarray}
\label{sum-B-C1}
\nonumber
 \sum _{ k \in B \backslash C_1} \alpha _k & \le & \sum _{ k \in (B \backslash C_1)
\cap D_1} \alpha _k + \sum _{ k \in B \backslash (C_1 \bigcup D_1)  } \alpha _k\\
 &\le& \sum _{ k \in (B \backslash C_1) \cap D_1}  \alpha _k  +  \sum_{ k \in B_1 \backslash D_1 }
\alpha _k \sem \text{(by \eqref{B-B1-C1})}  \\
\nonumber  & \le & \sum _{ k \in (B \backslash C_1) \cap D_1} \alpha _k + \varepsilon_1   \sem \text{(by
\eqref{sum-D1}).}
\end{eqnarray}

Now we take the set $C$ given by $C= (D_1 \cap F_1) \bigcup \bigl( (B \backslash B_1) \cap
F_1\bigr) \bigcup \bigl( (B \backslash C_1) \cap D_1)$. Let us notice  that in view of
\eqref{B-B1-C1} the three subsets  whose union is $C$ are pairwise disjoint.

We deduce that
\begin{eqnarray*}
\nonumber
 \sum _{ k \in C} \alpha _k & = & \sum _{ k \in D_1 \cap F_1} \alpha _k + \sum _{ k \in (B \backslash B _1) \cap F_1 }
  \alpha _k   +   \sum _{ k \in (B \backslash C _1) \cap D_1 }   \alpha _k   \\
\nonumber &\ge& \sum _{ k \in B_1 \cap C_1} \alpha _k   + \sum _{ k \in B \backslash B _1}
  \alpha _k   +   \sum _{ k \in B \backslash C _1  }   \alpha _k - 4 \varepsilon_1 \sep
\text{(by \eqref{sum-B1-C1}, \eqref{sum-B-B1} and \eqref{sum-B-C1})}  \\
\nonumber & = & \sum _{ k \in B} \alpha _k -4   \varepsilon_1 \\
& > & 1- \frac{4\varepsilon_0}{r} - 4\varepsilon_1  \sem \text{(by \eqref{sum-B-alfa})} \\
& > & 1- \varepsilon.
\end{eqnarray*}

If $D_1= \varnothing$, then  $C= (B \backslash B_1 ) \cap F_1$. In this case we choose any
elements $u_0 \in S_M $ and $m_{1}^* \in S_{M^*}$ with $m_{1}^* (u_0)=1.$ Analogously, in case
that $F_1= \varnothing$, we have $C= (B \backslash C_1 ) \cap D_1$ and we  choose $v_0\in S_N$
and $n_{1}^* \in S_{ N^*}$ such that $n_{1}^* (v_0)=1$. Otherwise  $D_1 \ne \varnothing$ and $F_1
\ne \varnothing$ and so the elements $m_{1}^* $ and $n_{1}^* $ satisfying \eqref{uk-mk} and
\eqref{vk-nk} attain their norms; so in this case we can choose $u_0 \in S_M$ and $v_0\in S_N$
with $m_{1}^*(u_0)=1$ and $n_{1}^* (v_0)=1$.

For  each  $k \in C$ we define
$$
z_k = \left\{%
\begin{array}{ll}
    r_k u_k + s_k v_k  & \ \hbox{if} \ k \in D_1 \cap F_1  \\
    r_k u_0 + s_k v_k  & \ \hbox{if} \ k \in (B\backslash B_1 )\cap F_1  \\
    r_k u_k + s_k v_0  & \ \hbox{if} \ k \in (B\backslash C_1) \cap D_1.  \\
\end{array}%
\right.
$$

We claim that $\Vert z_k - x_k \Vert < \varepsilon$ for each $k \in C$. To see this
observe that for $k \in D_1 \cap F_1$ we have
\begin{eqnarray*}
\Vert z_k - x_k \Vert &\le & \Vert z_k - y_k \Vert + \Vert y_k - x_k \Vert\\
&\le& \Bigl \vert \Bigl(  r_k \Bigl \Vert u_k - \frac{ m_k }{r_k } \Bigr \Vert, s_k \Bigl \Vert v_k - \frac{
n_k }{ s_k } \Bigr \Vert \Bigr) \Bigr \vert +
2\varepsilon_0    \sem \text{(by \eqref{y-x})} \\
&\le& \bigl \vert \bigl( r_k \varepsilon_1, s_k \varepsilon_1 \bigr) \bigr \vert + 2\varepsilon_0
\sem \text{(by \eqref{uk-mk} and \eqref{vk-nk})} \\
&\le & \varepsilon_1+2\varepsilon_0 < \varepsilon. \\
\end{eqnarray*}

For $k \in (B\backslash B_1) \cap F_1$ we have that
\begin{eqnarray*}
\Vert z_k - x_k \Vert &\le & \Vert z_k - y_k \Vert + \Vert y_k - x_k \Vert\\
&\le&   2 r_k + s_k \Bigl \Vert v_k - \frac{ n_k}{ s_k } \Bigl \Vert +
2\varepsilon_0   \sem \text{(by \eqref{y-x})} \\
&\le&   2 s + \varepsilon_1 + 2\varepsilon_0  \sem \text{(by \eqref{vk-nk})}  \\
&< & \varepsilon. \\
\end{eqnarray*}

In case when $k \in (B\backslash C_1) \cap D_1$,
\begin{eqnarray*}
\Vert z_k - x_k \Vert &\le & \Vert z_k - y_k \Vert + \Vert y_k - x_k \Vert\\
&\le&  r_k \Bigl \Vert u_k - \frac{ m_k}{ r_k } \Bigl \Vert + 2 s_k +
2\varepsilon_0  \sem \text{(by \eqref{y-x})} \\
&\le&  \varepsilon_1 +2s + 2\varepsilon_0 \sem \text{(by \eqref{uk-mk})}  \\
&< & \varepsilon \\
\end{eqnarray*}
and this proves the claim.

\vskip-2mm

Now we observe that  $\alpha m_{1}^* + \beta n_{1}^* \in X^*$ and $\Vert \alpha m_{1}^* + \beta
n_{1}^* \Vert = \vert (\alpha, \beta ) \vert ^* =1$. In view of \eqref{uk-mk}, \eqref{vk-nk} and
the choice of $u_0 $ and $v_0$, for  each   $k \in C$ one clearly has
\begin{eqnarray*}
(\alpha m_{1}^* + \beta n_{1}^*)(z_k) &= &  \alpha  m_{1}^* (P(z_k)) + \beta  n_{1}^* (Q(z_k))   \\
&=&  \alpha r_k + \beta s_k = 1.
\end{eqnarray*}
\end{proof}

Let us remark that we have been informed by the referee about the paper by 
	F.J. Garc{\'i}a-Pacheco \cite{Ga}, where the easier part of the above result was independently obtained.

Before we state and prove a~stability result of AHSp for some infinite sums of Banach spaces that
includes infinite $\ell_p$-sums,  we recall the following notion that was introduced in
\cite[Definition 2.1]{CKLM}.

\begin{definition}
\label{def-AHP}
A Banach space $X$ has the  {\it approximate hyperplane property} (AHp) if there exists a
function $\delta: (0,1) \longrightarrow (0,1)$ and a $1$-norming subset $C$ of $S_{X^*}$
satisfying the following property.

Given   $\varepsilon> 0$ there is a function $\Upsilon_{X, \varepsilon}: C \llll
S_{X^*}$ with  the following condition
$$
x^* \in C,\  x \in S_X, \ \Rea x^* (x) > 1 - \delta (\varepsilon) \ \Rightarrow \
\dist(x, F(\Upsilon_{X, \varepsilon} (x^*))) < \varepsilon,
$$
where $F(y^*) = \{ y \in S_X : \Rea y^* (y)=1\}$ for any $y ^*\in S_{X^*}$.

A family of Banach spaces $\{ X_i : i \in I\}$   has {\it $AHp$ uniformly}  if every space $X_i$
has property $AHp$ with the same function $\delta$.
\end{definition}

Clearly we can assume that the $1$-norming subset $C$ in the previous definition satisfies
$\mathrm{T}C \subset C$, where $\mathrm{T}$  is the unit sphere of the scalar field.

Let us notice that a similar property to AHp was implicitly used to prove that several classes of
spaces have AHSp (see \cite{AAGM1}).

It is known that property  AHp implies AHSp (see for instance \cite[Proposition 2.2]{CKLM}).
Examples of spaces having  AHp  are finite-dimensional spaces, uniformly convex spaces, $L_1
(\mu)$ for every measure $\mu$   and also $C(K)$ for every compact Hausdorff topological space
$K$ (see \cite[Propositions 3.5, 3.8, 3.6 and 3.7]{AAGM1} and also \cite[Corollary 2.12]{CKLM}).

In what follows we will use the standard notation from the theory of Banach lattices as presented
for example in \cite{LT}. We denote by $\omega$ the space of all  real sequences. As usual, the
order $|x|:=(|x_n|) \leq |y|$ for $x=(x_n)$, $y=(y_n)\in \omega$ means that $|x_n| \leq |y_n|$
for each $n \in \mathbb{N}$.

A (real) Banach space $E\subset \omega$ is  {\it solid} whenever $x \in w$, $y \in E$ and $|x|
\leq |y|$ then $x \in E$ and $\|x\|_E \leq \|y\|_E$. $E$ is said to be a~\emph{Banach sequence
lattice} (or Banach sequence space)  if  $E\subset \omega$, $E$ is solid and there exists $u \in
E$ with $u>0$. A~Banach sequence lattice $E$ is said to be \emph{order continuous} if for every
$0 \le f_n \downarrow 0$, it follows that $\|f_n\|_E \to 0$. If $E$ is an order-continuous Banach
sequence lattice, then $E^{*}$ can be identified in a~natural way with the \emph{K\"othe dual}
space $(E', \|\cdot\|_{E'})$ of all $x=(x_k)\in \omega$ equipped with the norm
\[
\|x\|_{E'} := \sup_{(y_k)\in B_E} \sum_{k=1}^{\infty} |x_k y_k|\,.
\]

Let $E$ be a~Banach sequence lattice. For a~given sequence $(X_k,
\|\cdot\|_{X_k})_{k=1}^{\infty}$ of Banach spaces the vector space of sequences $x =
(x_k)_{k=1}^{\infty}$, with $x_k\in X_k$ for each $k\in \mathbb{N}$ and with $(\|x_k\|)\in E$,
becomes a~Banach space when equipped with the norm
\[
\|(x_k)\| = \big\|\big(\|x_k\|_{X_k}\big)\big\|_{E};
\]
this space will be denoted by $\big(\oplus \sum_{k=1}^{\infty} X_k\big)_E$.

Finally we recall that a Banach lattice  $E$ is {\it uniformly monotone} (UM) if for every
$\varepsilon > 0$ there is $\delta> 0$ such that whenever $x \in S_E$, $ y \in E$ and $x,y \ge 0$
the condition $\Vert x +y \Vert \le 1 + \delta $ implies that $\Vert y \Vert \le \varepsilon$. It
is known that every UM Banach lattice is order continuous (see \cite[Theorem 22]{Bi}).

We will use the following duality result which is well known in the case $E=\ell_p$ with $1\leq
p<\infty$ or $E=c_0$ (see, e.g., \cite[Theorem 12.6]{AB}). Since the proof of the general case is
similar we omit it.

\begin{theorem}
\label{th-dual}
Let $E$ be an order continuous Banach sequence lattice and let $(X_n)$ be
a~sequence of Banach spaces. Then the mapping $\big(\oplus
\sum_{n=1}^{\infty}X_n^{*}\big)_{E'} \ni x^{*}=(x_n^{*}) \mapsto \phi_{x^{*}}$
defined by
\[
\phi_{x^{*}}(x_n) = \sum_{n=1}^{\infty} x_n^{*}(x_n), \quad\, (x_n) \in \Big(\oplus
\sum_{n=1}^{\infty} X_n\Big)_E.
\]
is  an isometrical isomorphism from $\big(\oplus \sum_{n=1}^{\infty}
X_n^{*}\big)_{E'}$ onto $\big(\big(\oplus\sum_{n=1}^{\infty} X_n)_{E}\big)^{*}$.
\end{theorem}

The following technical result will be useful.

\begin{lemma}
\label{norming} Let $E$ be a~Banach sequence  lattice which is
order continuous and $\{X_k: k \in \N\}$ be a sequence of
$($nontrivial$)$ Banach spaces. For each natural number $k$ assume
that $C_k\subset S_{X_k^{*}}$ is~a $1$-norming set for $X_k$. Then
the set $C$ given by
$$
C = \{ ( e_ k^* \lambda _k x_k^{*} ): e^{*} \in S_{E^\prime} ,  e^* \ge 0,  \lambda _k \in \K,
\vert \lambda_k \vert=1, x_k^{*} \in C_k,\forall k \in \N\}
$$
is a subset of $ S_{Z^*}$,  a~$1$-norming set for $Z$, where $\K$ is the scalar field and
$Z=\bigl( \oplus\sum_{k=1}^\infty X_k \bigr)_E$.
\end{lemma}

\begin{proof}
By  Theorem \ref{th-dual} the set $C$ is contained in $S_{ Z^* }$. Let $ z=  (z_k) \in Z$ and $
\varepsilon>0$. By assumption we know that $( \Vert z_k\Vert ) \in E$. In view of Theorem
\ref{th-dual}, $E^* $ coincides with $E^\prime $, so there is a nonnegative  element $e^*\in S_
{E^\prime }$ such that $ e^* \bigl( (\Vert z_k\Vert) \bigr)= \Vert ( \Vert z_k\Vert ) \Vert_E =
\Vert z \Vert $. For each $k \in \N$,  $C_k$ is a~$1$-norming set for $X_k$ and so  there exists
$ z _k  ^* \in C _k $ and a~scalar $\lambda _k $ with $\vert \lambda _k \vert=1$ such that $\rea
\lambda_k z_k^* (z_k)> \Vert z_k \Vert - \dfrac{\varepsilon}{ (e_k^* +1) 2^k}$.     The element
$z^*= \bigl( e_k^* \lambda_k z_k^* \bigr) \in C $ and
$$
\rea  \; z^*(z)= \sum_ {k=1}^\infty \rea \; e_k^{*}\,\lambda_{k}\,
z_k^*(z_k)>\sum_ {k=1}^\infty e_k^* \Bigl( \Vert z_k \Vert -
\frac{\varepsilon}{(e_k^* +1) 2^k} \Bigr) \ge \Vert z \Vert -
\varepsilon .
$$
We proved that $C$ is a~$1$-norming set for $Z$.
\end{proof}

Now we are ready to prove the stability of the AHSp.

\begin{theorem}
\label{th-AHSP-sequence}
Let $E$ be a~Banach sequence lattice with the AHSp and such that it is uniformly monotone. Assume
that $\{ X_{k}: k \in \N\}$  has property AHp uniformly. Then the space $\big(\oplus
\sum_{k=1}^{\infty} X_k\big)_E$ has the AHSp.
\end{theorem}

\begin{proof}
We take $M= \{ k \in \N: X_k \ne 0\}$.  If $M$ is infinite, there is no loss of generality in assuming that
$M=\N$. Otherwise the proof of the statement is essentially the same but easier.

So  we  assume that $X_k \ne \{0\}$ for each $k$. We put $Z:= \big(\oplus \sum_{k=1}^{\infty} X_k \big)_E$.

Let us fix  $0 < \varepsilon < 1$. By assumption, $\{X_k: k \in \N \}$  has AHp uniformly, so there is
$\delta :(0,1) \llll (0,1)$ satisfying Definition \ref{def-AHP} for each $k \in \N$. We choose $0 < \eta <
\min \bigl\{ \frac{\varepsilon}{4}, \delta (\frac{\varepsilon}{4} )  \bigr\}$. Since $E$ is uniformly
monotone,  we can use condition ii) in \cite[Theorem 6]{HKM}, so there is $0 < \alpha <\varepsilon/4 < 1$
satisfying that
\begin{equation}
\label{um-E}
 e \in  S_E,  \sep e\ge 0, \sep A \subset \N,  \sep  \bigl \Vert e \chi _ A \bigr \Vert_E
> \frac{\varepsilon}{4} \ \Rightarrow \ \bigl \Vert e \chi _{ \N \backslash A } \bigr \Vert_E < 1 -
\alpha .
\end{equation}

For  $r= (1+ 2\eta - \alpha \eta)/(1+2\eta)$, we choose $0 < \varepsilon^\prime < (1-r)\varepsilon/ 3 $. Then
by  our assumption, it follows that there is $0 < \eta^\prime < \varepsilon^\prime$ such that $E$ satisfies
the statement (d) in Proposition \ref{char-AHSP} for $(\varepsilon^\prime, \eta^\prime).$

In order to prove that $Z$  satisfies the AHSp we will show that condition (d) in Proposition \ref{char-AHSP}
is satisfied for $(\varepsilon,\eta^\prime)$.

Assume that $\bigl( z_n \bigr) $ is a~sequence in $S_Z$ and $\sum \alpha _n$  is a~convex series such that
\linebreak[4]
$\displaystyle{ \biggl \Vert \sum_{n=1}^\infty \alpha _nz _n \biggr \Vert > 1 - \eta^\prime}$.

Then
\begin{eqnarray}
\label{1-eta}
\nonumber 1 - \eta^\prime & < & \biggl \Vert  \sum_{n=1}^\infty  \alpha _n z_n \biggr \Vert  \\
\nonumber & = &
\biggl \Vert  \biggl( \Bigl \Vert \sum _{n=1}^\infty \alpha _n  z_n(k) \Bigr \Vert \biggr)_{k} \biggr \Vert_E \\
& \le &
\biggl \Vert  \biggl(  \sum _{n=1}^\infty \alpha _n \bigl \Vert z_n(k) \bigr \Vert \biggr)_{k} \biggr \Vert_E\\
\nonumber & =  & \biggl \Vert   \sum _{n=1}^\infty \alpha _n  \Bigl(\bigl \Vert z_n(k) \bigr
\Vert \Bigr)_{k} \Bigr \Vert_E. \nonumber
\end{eqnarray}

Combining our hypothesis that $E$ has the AHSp with $\bigl(  \Vert z_n(k) \Vert \bigr)_k \in S_E$ for each
positive integer $n$, we conclude that there is a~finite subset $A \subset \N$ and $\{r_{n}: n \in A\}
\subset S_E$ such that
\begin{equation}
\label{sum-A-big}
\sum _{n \in A} \alpha _n > 1 - \varepsilon^\prime
\end{equation}
and also
\begin{equation}
\label{E-AHSP}
 r_n \ge 0,  \  \Vert r_n - \bigl( \Vert z_n (k) \Vert\bigr) _k \Vert_E < \varepsilon^ \prime
 \sep
\text{and there is} \sep r^* \in S_{ E^\prime} \,\,\text{ with }\,\, r^* (r_n) = 1, \, \text{ for
all \, } n\in A.
\end{equation}
Hence from \eqref{1-eta} and \eqref{sum-A-big} we obtain that
\begin{equation}
\label{sum-A-alfa-z}
1-\eta^\prime  - \varepsilon^ \prime < \biggl \Vert \sum _{ n \in A} \alpha _n z_n
\biggr \Vert .
\end{equation}

For each $k \in \N$ we choose an element $x_k \in S_{X_k}$ and define for every $n \in A$ the element $u_n$
in $Z$ given by
$$
u_n(k) = \left\{%
\begin{array}{ll}
    r_n(k) \dfrac{ z_n(k) }{ \Vert z_n(k) \Vert  } & \hbox{ if } z_n(k) \ne 0 \\
     &  \\
    r_n(k)x_k & \hbox{otherwise.} \\
\end{array}%
\right.
$$
By \eqref{E-AHSP} it is clearly satisfied that
\begin{equation}
\label{un-zn}
\bigl \Vert u_n - z_n \bigr \Vert =   \bigl \Vert r_n - \bigl( \Vert z_n (k) \Vert \bigr)_k \bigr
\Vert_E < \varepsilon ^\prime, \quad\,  \sem \forall n \in A.
\end{equation}

So in view of \eqref{sum-A-alfa-z} we obtain that
\begin{equation}
\label{sum-A-alfa-u}
1-\eta^\prime  - 2\varepsilon^ \prime < \biggl \Vert \sum _{ n \in A} \alpha _n u_n
\biggr \Vert .
\end{equation}

By assumption, $\{X_k: k \in \N \}$  has AHp uniformly.   For each $k \in \N$  let $G_k \subset S_{ X_{k}^*}$
be the $1$-norming set for $X_k$  satisfying Definition \ref{def-AHP}. We can also assume that  $G_k=
\{\lambda x^* : \lambda \in \K, \vert \lambda \vert=1, x^* \in G_k\}$ for each $k \in \N$. By Lemma
\ref{norming}  there is $z^* \in S_{Z^*}$ that can be written as $z^*\equiv \bigl( z_{k}^* \bigr)= \bigl(
e^*_k  x_{k}^*\bigr)$ where $e^* \in S_{E^\prime}$,$e^* \ge 0$  and $   x_{k}^* \in G_k$ for each $k \in \N$
satisfying that
$$
1- \eta ^\prime - 2\varepsilon^ \prime <  \Rea z^* \biggl( \sum _{ n \in A} \alpha
_n u_n \biggr).
$$

Now we define the set $C$ by $C=\bigl\{ n\in A : \rea z^* (u_n) > r \bigr\}$. By Lemma \ref{elemental} we
obtain that
\begin{equation}
\label{sum-C-big}
\sum _{n \in C} \alpha _n >  1 - \frac{ \eta^\prime  + 2 \varepsilon^\prime}{1-r} > 1 - \varepsilon > 0.
\end{equation}

For each element   $n \in C$ we have that
\begin{eqnarray}
\label{r-z*-un}
\nonumber r &< &   \Rea z^* \bigl(  u_n \bigr) =  \sum_{k=1}^\infty  \Rea   z_{k}^* (u_n(k))  \\
&\le & \sum_{k=1}^\infty  \bigl\vert   z_{k}^* (u_n(k))  \bigr\vert \\
\nonumber &\le & \sum_{k=1}^\infty  \bigl\Vert   z_{k}^* \bigr\Vert \bigr\Vert    u_n(k)  \bigr\Vert \\
\nonumber &\le &  \bigl \Vert  \bigl( \bigl \Vert  z_{k} ^* \bigr\Vert \bigr) \bigr\Vert
_{E^\prime} \bigl \Vert  \bigl( \bigl \Vert  u_n(k)  \bigr\Vert \bigr)_k \bigr\Vert _{E} \\
\nonumber&= &1.
\end{eqnarray}
For each  $n \in C$  and $k \in \N$ we  put
$$
d_n(k)= \Vert z_{k}^* \Vert \Vert u_n(k) \Vert - \Rea z_{k}^*  (u_n(k) ).
$$
The chain of inequalities  \eqref{r-z*-un}  implies that
\begin{equation}
\label{sum-dn-k} \sum _{k =1}^\infty d_n(k) \le 1-r , \sem \forall
 n \in C.
\end{equation}

We now fix a positive integer $k$. If $z_{k}^*=0$, then $d_n(k)=0$ for every $n \in C$. If $n \in C$ and
$u_n(k)=0$ for some $k \in \N$ then  $d_n(k)=0$. Otherwise it is satisfied that
\begin{equation}
\label{xk*-unk}
\Rea \frac{ z_{k}^*}{ \Vert  z_{k}^* \Vert } \biggl(  \frac{ u_n(k) }{ \Vert u_n(k)
\Vert }  \biggr)  = 1 -
 \frac{ d_n(k) } { \Vert z_{k}^* \Vert \Vert  u_n(k) \Vert }.
\end{equation}

In what follows, for each $n \in C$, we consider the following subset
$$
B_n = \bigl\{ k \in \N: d_n(k) <  \eta \Vert z_{k}^* \Vert \Vert
u_n(k) \Vert \bigr\}.
$$
By \eqref{r-z*-un}  we know that
\begin{eqnarray*}
\label{r-Bn}
 r &< &     \sum_{k=1}^\infty   \Vert z_{k}^*  \Vert \Vert u_n(k) \Vert \\
 &&   \\
&=&  \sum_{k\in B_n}  \Vert z_{k}^*  \Vert \Vert u_n(k) \Vert  +  \sum_{k\in \N \backslash B_n}  \Vert z_{k}^*  \Vert \Vert u_n(k) \Vert    \\
 & & \\
&\le & \sum_{k\in B_n}  \Vert z_{k}^*  \Vert \Vert u_n(k) \Vert  + \dfrac{1}{\eta}  \sum_{k\in \N \backslash B_n} d_n(k)     \\
& & \\
 &\le & \sum_{k\in B_n}  \Vert z_{k}^*  \Vert \Vert u_n(k) \Vert  + \dfrac{1}{\eta} (1-r) \sem \text{(by \eqref{sum-dn-k})} .
\end{eqnarray*}
As a consequence,
\begin{equation}
\label{sum-Bn-norms}
\sum_{k\in B_n}  \Vert z_{k}^*  \Vert \Vert
u_n(k) \Vert  > r- \frac{1-r}{\eta} > 0
\end{equation}
and  in view of \eqref{r-z*-un} we deduce that
\begin{equation}
\label{sum-N-Bn-norms}
 \sum_{k\in  \N \backslash B_n} \Vert z_{k}^* \Vert \Vert u_n(k) \Vert  <  1- r+
\frac{1-r}{\eta}, \sem \forall  n \in C.
\end{equation}

In view of \eqref{xk*-unk}, for every $n \in C$ and $k \in B_n$ it is satisfied that
$$
\Rea \; x_{k}^* \biggl( \frac{ u_n(k) }{ \Vert u_n(k)\Vert }
\biggr) = \Rea \frac{ z_{k}^* }{ \Vert   z_{k}^* \Vert  } \biggl(
\frac{ u_n(k) }{ \Vert u_n(k)\Vert } \biggr)=1 - \frac{d_n(k)}{
\Vert z_{k}^* \Vert \Vert  u_n(k) \Vert }
>  1 - \eta.
$$
Now we will use that for each $k$  the space $X_k$ has the property AHp for the function $\delta $,  $\eta <
\delta \bigl( \frac{\varepsilon }{4}\bigr) $ and  $ x_{k}^*  \in G_k$. Hence for each $k \in \cup _{l\in C}
B_l$, there is $y_{k}^* \in S_{X_{k}^*}$  such that if $n \in C $ and $k \in B_n$ there is $m_n(k) \in
S_{X_k}$ with
\begin{equation}
\label{X-P}
 \biggl \Vert m_n (k) - \frac{ u_n(k)}{ \Vert   u_n(k)
\Vert } \biggr \Vert < \frac{\varepsilon}{4}, \sem \text{and}\sem \Rea y_{k}^*
(m_n(k))=1, \sem  \forall n \in C, \sep \forall k  \in B_n.
\end{equation}
Let $D= \mathbb{N} \setminus \bigcup_{l \in C} B_l$. For each $k \in D$, we choose any element $y_{k}^* \in
S_{X_k^*}$ such that $y_{k}^* (x_k)=1$.

For each $n \in C$, we write $C_n= \bigcup_{l \in C} B_l
\backslash B_n$ and  define $v_n \in Z$ by
$$
v_n (k) = \left\{%
\begin{array}{ll}
r_n(k) m_n(k)  & \ \hbox{if}\ k \in B_n \\
r_n(k) m_{p(k)}(k)  & \ \hbox{if}\ k \in C_n \\
r_n(k)  x_k  & \ \hbox{if}\ k \in D, \\
\end{array}%
\right.
$$
where $p(k)= \min \bigl\{ s \in C: k \in B_s \bigr\}$ if $k \in \bigcup _{ l \in C} B_l.$ It is clear that
$\Vert v_n \Vert = \Vert r_n \Vert_E =1$ for each $n \in C$.

We clearly have that
\begin{eqnarray}
\label{un-Bn}
\nonumber
\Vert r_n \chi_{B_n} \Vert_E  &= &  \Vert u_n  \chi_{B_n} \Vert  \\
\nonumber
 &\ge &  \Rea z^* \bigl( u_n  \chi_{B_n}  \bigr) \\
\nonumber
 &= & \Rea     \sum_{k\in B_n }   z_{k}^*  \bigl( u_n(k) \bigr)  \\
\nonumber &= &     \sum_{k=1 }  ^\infty   \Rea   z_{k}^*  \bigl( u_n(k) \bigr)
-  \sum_{k\in \N \backslash B_n }   \Rea   z_{k}^*  \bigl( u_n(k) \bigr)  \\
&>&  r - \sum_{k\in \N \backslash B_n } \Rea   z_{k}^*  \bigl( u_n(k) \bigr) \sem \text{(by \eqref{r-z*-un})} \\
\nonumber
&\ge &  r - \sum_{k\in \N \backslash B_n }   \Vert    z_{k}^* \Vert \Vert u_n(k) \Vert   \\
\nonumber
&>&  r -   \Bigl(1-r + \frac{1-r}{\eta}\Bigr) \sem \text{(by \eqref{sum-N-Bn-norms})} \\
\nonumber &=&  2r -1 - \frac{1-r}{\eta} = 1 - \alpha.
\end{eqnarray}

Since $0\leq r_n $ for each $n \in C$ and $\{r_n:n \in C \} \subset  S_E$, from  \eqref{un-Bn} and
\eqref{um-E} it follows that
\begin{equation}
\label{rn-comp-Bn}
\bigl \Vert  r_n \chi_{ \N \backslash B_n } \bigr \Vert_E \le \frac{ \varepsilon}{4}.
\end{equation}
\noindent For every $n \in C$  and $k \in B_n$, in view of
\eqref{X-P} we have that
\begin{eqnarray}
\label{vn-un-k}
 \Vert v_n(k) - u_{n}(k) \Vert &= &  \  \bigl \Vert  r_n(k)
m_n(k)  -  u_{n}(k) \bigr \Vert \\
\nonumber &\le& \frac{\varepsilon}{4}  r_n(k).
\end{eqnarray}

Hence from \eqref{vn-un-k}, for every $n \in C$  we have that
\begin{eqnarray*}
\Vert v_n - u_{n}  \Vert &\le  &  \  \bigl \Vert \bigl( v_n - u_n \bigr) \chi_{B_n} \bigr \Vert +
  \bigl \Vert v_n \chi_{ \N \backslash B_n } \bigr \Vert+ \bigl \Vert u_n \chi_{ \N \backslash B_n } \bigr \Vert \\
&\le&   \frac{\varepsilon}{4} \bigl \Vert  r_n \bigr \Vert _ E +  2\bigl \Vert r_n \chi_{
\N \backslash B_n } \bigr \Vert_E    \sem \text{(by \eqref{vn-un-k})} \\
&\le &     \frac{3\varepsilon}{4}  \sem   \text{(by \eqref{rn-comp-Bn})}.
\end{eqnarray*}

\noindent
Combining with \eqref{un-zn}, we conclude that for each $n \in C$,
\begin{eqnarray*}
\Vert v_n - z_n  \Vert &\le  &  \  \Vert   v_n - u_n  \Vert +  \Vert u_n  - z_n \Vert \\
&\le&   \frac{3 \varepsilon}{4} +  \varepsilon^\prime   \\
&<&  \varepsilon .
\end{eqnarray*}

Let  $v^* $ be the element in $ Z^*$ given by $v^*= \bigl\{ r_k^{*} y_{k}^* \bigr\} $. By Theorem
\ref{th-dual} it is satisfied that
\linebreak[4]
 $ \Vert v^* \Vert = \Vert r^* \Vert_{E^\prime }=1.$ For  each $n \in C$ we clearly have that
\begin{eqnarray*}
v^* (v_n) &=  &  \sum_{k=1}^\infty
 r_{k}^{*}\,y_{k}^* (v_n(k))   \\
&= &   \sum_{k\in B_n} r_{k}^{*}\,r_n(k)\,y_{k}^* (m_n(k)) +
\sum_{k\in C_n} r_{k}^{*} r_n(k) y_{k}^{*} (m_{p(k)}(k))+
\sum_{k\in D} r_{k}^{*} r_n(k) y_{k}^{*}(x_k)        \\
&= &    \sum_{k=1} ^\infty    r_{k}^{*} r_n(k)
\sem  \text{(by \eqref{X-P})} \\
&= &  r^{*}(r_n) =1 \sem  \text{(by \eqref{E-AHSP})}.
\end{eqnarray*}
From \eqref{sum-C-big} we also know that $\sum _{n \in C} \alpha _n >   1 - \varepsilon$, so the proof is
finished.
\end{proof}

As we mentioned above uniformly convex spaces have AHp. Indeed in this case the modulus  of convexity  plays
the role of the function $\delta $  satisfying Definition \ref{def-AHP} and the identity function on the unit
sphere of the dual plays the role of the function $\Upsilon_\delta $ \cite[Lemma 2.1]{ABGM-tams}. So a family
$\{ X_i:i \in I\}$ of uniformly convex Banach spaces has the AHp uniformly in case that $\inf \{ \delta _i
(\varepsilon) : i \in I\} > 0$, for any $\varepsilon > 0$,  being $\delta _i$ the modulus of convexity of
$X_i$. Also $C(K)$ spaces and $L_1 (\mu)$ have AHp uniformly  for any compact Hausdorff space $K$ and any
measure $\mu$ \cite[Corollary 2.8]{CKLM}. As a consequence of Theorem \ref{th-AHSP-sequence} and
\cite[Theorem 4.1]{AAGM1} we deduce, for instance, the following result.

\begin{corollary}
\label{cor-sums}
Let $\{ X_k: k \in \N \}$ be a sequence of (nontrivial) Banach spaces  such that any of them is either a
uniformly convex space or  $C(K)$ (some compact $K$) or $L_1(\mu)$ (some measure $\mu$). Let  $A= \{ k \in
\N: X_k \ \text{\rm is a uniformly convex space} \} $ and assume that $\inf \{ \delta _k(\varepsilon):\,k \in
A\} > 0$ for every $\varepsilon > 0$,  being $\delta _k$ the modulus of convexity of $X_k$.  Then the pair
$\big(\ell_1, \big(\oplus \sum_{k=1}^{\infty} X_k\big)_{\ell_p} \big)$ satisfies the BPBp for every $1 \le p
< \infty$.
\end{corollary}

Let us remark that in general AHSp is not stable under infinite $\ell_1$-sums (see \cite[Corollary
4.6]{ACKLM}). So in order  to have the stability result  in Theorem \ref{th-AHSP-sequence} some additional
restriction is needed.  Now we show the following  partial converse of Theorem \ref{th-AHSP-sequence} that
extends to some infinite sums the necessary condition  obtained in Theorem \ref{th-estable}.

\begin{proposition}
\label{th-estable-E}
Let $\{ X_k: k \in \N \}$ be a sequence of (nontrivial) Banach spaces  and $E$ be an order
continuous Banach sequence lattice.  Assume that the space $Z=\big(\oplus \sum_{k=1}^{\infty} X_k
\big)_E$ has the approximate hyperplane series property. Then there is a function $\tilde {\eta}:
(0,1) \to (0,1)$ such that $X_k$ satisfies the approximate hyperplane series property with the
function $\tilde{\eta}$ for every $k \in \N$.  More precisely, one can take the function given by
$\tilde{\eta} \bigl( \varepsilon \bigr)= \eta \bigl( \frac{\varepsilon}{2} \bigr)$, where $\eta$
is the function satisfying Definition {\rm{\ref{def-AHSP}}} for $Z$.
\end{proposition}
\begin{proof}
It suffices to prove that  $X_1$ has the property AHSp for $\tilde{\eta}$. Consider the subspace
$Z_1$ of $Z$ given by
$$
Z_1 =\{ z \in Z  : z(k)=0, \forall k  \ge 2 \}.
$$
Notice that the mapping  from $Z_1$ into $X_1$ given by  $ z \mapsto z(1) \Vert e_{1} \Vert _E $ is a linear
isometry, where $e_1$ is the sequence given by $e_{1}(k)= \delta _{1}^k$ for each natural number $k$. Since
AHSp is clearly preserved by linear isometries (and the function $\eta $ satisfying AHSp also) then it
suffices to prove that  $Z_1$ satisfies AHSp with the function $\tilde{\eta}$.

So let us fix $0 < \varepsilon < 1$. Assume that   $\alpha_n \ge 0$, $u_n \in S_{Z_1}$ for every
$n$,  $\sum _{n=1}^\infty \alpha_n=1 $ and it is also satisfied that
$$
\Bigl  \Vert \sum _{n=1}^\infty \alpha _n u_n \Bigr \Vert > 1 - \eta \Bigl(
\frac{\varepsilon}{2} \Bigr) .
$$
By assumption $Z$ has the AHSp, so  there is a subset $A \subset \N$ such that $\sum _{n\in A} \alpha _n > 1
- \frac{\varepsilon}{2} > 1 - \varepsilon$, $z^* \in S_{Z^*}$  and $ \{z_n: n \in A \} \subset S_{Z} $ such
that
\begin{equation}
\label{zn-un}
\Vert  z_n - u_n \Vert < \frac{\varepsilon}{2}  \sem \text{and} \sem   z^* (z_n)= 1, \seg \forall n \in A.
\end{equation}
For every $n \in A $ we define the element $y_n\in Z_1 $ given by
$$
y_n (1)= z_n(1), \sem y_n (k)=0, \sem \forall k \ge 2.
$$
Let us fix $n \in A$. We clearly have that
\begin{equation}
\label{yn-un}
 \Vert y_n - u_n \Vert = \bigl\Vert \bigl( \Vert y_n(k)- u_n (k)\Vert \bigr)   \bigr \Vert _E \le
 \bigl\Vert \bigl(  \Vert z_n(k)- u_n (k)\Vert \bigr) \bigr\Vert _E = \Vert z_n - u_n \Vert  <
 \frac{\varepsilon}{2}.
\end{equation}
Since we  know that
$$
\Vert y_n  \Vert  \le \Vert z_n \Vert =  1, \sem \forall n \in A,
$$
in view of \eqref{yn-un} we deduce that
\begin{equation}
\label{norm-yn}
1- \frac{\varepsilon }{2} \le \Vert y_n  \Vert  \le 1, \sem \forall n \in A.
\end{equation}
As a consequence  of Theorem \ref{th-dual} we know that $z^*  \in \big(\oplus
\sum_{k=1}^{\infty}X_k^{*}\big)_{E'} $  and we also   have
\begin{equation}
\label{z*-yn-1}
z^*(1)  ( y_n (1) ) = z^*(1)  ( z_n (1) ) =     \Vert z^*(1)  \Vert\; \Vert  z_n(1) \Vert= \Vert z^*(1)
\Vert\; \Vert  y_n(1) \Vert, \sem \forall n \in A.
\end{equation}

On the other hand, it is satisfied that
\begin{eqnarray}
\label{z*1-n-0}
\vert z^* (1)( y_n (1) ) \vert  &=&   \vert z^* (y_n)\vert  \nonumber \\
&\ge &  \vert z^* (z_n)\vert -  \vert z^* (y_n-z_n)\vert  \nonumber \\
&\ge & 1-  \Vert z_n- y_n \Vert  \nonumber\\
&\ge &  1-  \Vert z_n- u_n \Vert -  \Vert u_n- y_n \Vert \\
\nonumber
&\ge &  1-  2\Vert z_n- u_n \Vert \sem \text{ (by \eqref{yn-un})} \\
\nonumber
 & >  &   1-  \varepsilon > 0 \sem  \text{ (by \eqref{zn-un})}.
\end{eqnarray}

We denote by $w^*$ the element in $Z^*$ given by
$$
w^* (1)= z^*(1), \sem w^* (k)= 0, \sem  \text{if} \sep  k \ge 2 .
$$

Notice that $\Vert e_1 \Vert _{E^\prime}  \Vert e_1 \Vert _{E}=1$.  So it is clearly satisfied
\begin{eqnarray*}
\Rea w^*  (y_n) &  =  &  \Rea z^*  (y_n)  \\
& = &  \Vert z^*(1) \Vert\; \Vert  y_n(1) \Vert \sem \text{(by  \eqref{z*-yn-1})} \\
 & = &  \frac{\Vert w^* \Vert}{ \Vert e_1 \Vert_{E^\prime} } \frac{\Vert y_n
\Vert}{\Vert e_1 \Vert_E } \\
& = & \Vert w^* \Vert \Vert y_n \Vert,
\end{eqnarray*}
and bearing in mind  \eqref{z*1-n-0} we deduce that $w^*  (y_n) \ne 0$.

Since for each $n \in A$ we have also that
\begin{eqnarray*}
\Bigl\Vert u_n - \frac{y_n }{\Vert y_n \Vert } \Bigr\Vert &  \le  &  \Vert u_n - y_n \Vert +
\Bigl\Vert  y_n - \frac{y_n }{\Vert y_n \Vert } \Bigl\Vert \\
&  < &   \frac{\varepsilon}{2} + 1- \Vert y_n \Vert
\le  \varepsilon \sem \text{(by
\eqref{yn-un}  and \eqref{norm-yn})},
\end{eqnarray*}
we checked that $Z_1$ has the AHSp for the function $\tilde{\eta}$ as we wanted to show.
\end{proof}

\vspace{3mm}

{\bf Acknowledgements.}
	The third author is grateful to the Office of Graduate Studies of the University of Isfahan for their support.  The authors thank  the referee for providing the reference \cite{Ga}.

\end{document}